\pdfoutput=1
\documentclass[leqno]{article}
\usepackage[normal]{phaine_style}
\usepackage[margin=1.5in]{geometry}

\renewcommand{\Open}{\operatorname{Open}}

\newcommand{\SC}{\upbeta}

\renewcommand{\E}{\mathcal{E}}
\newcommand{\Crm}{\mathrm{C}}

\newcommand{\RHom}{\mathrm{RHom}}

\renewcommand{\st}{\mathrm{st}}
\newcommand{\PrLst}{\PrL_{\st}}

\newcommand{\BZZ}{\mathrm{B}\ZZ}

\newcommand{\Shpost}{\Sh^{\post}}

\newcommand{\Comp}{\categ{Comp}}
\newcommand{\Compop}{\Comp^{\op}}
\newcommand{\CompL}{\Comp_{/L}}
\newcommand{\CompK}{\Comp_{/K}}

\newcommand{\CompS}{\Comp_{/S}}

\newcommand{\CompX}{\Comp_{/X}}

\newcommand{\Extr}{\categ{Extr}}

\renewcommand{\X}{\Xcal}
\renewcommand{\Y}{\Ycal}

\newcommand{\Xhyp}{\X^{\hyp}}

\newcommand{\Xpost}{\X^{\post}}
\newcommand{\Ypost}{\Y^{\post}}

\newcommand{\LCH}{\categ{LCH}}
\newcommand{\LCHop}{\LCH^{\op}}

\DeclareMathOperator{\Clop}{Clop}

\DeclareMathOperator{\sheaf}{sheaf}
\DeclareMathOperator{\cond}{cond}
\newcommand{\Hcond}{\Hup_{\cond}}
\newcommand{\Hsheaf}{\Hup_{\sheaf}}
\newcommand{\RGammacond}{\Rup\Gamma_{\cond}}
\newcommand{\RGammasheaf}{\Rup\Gamma_{\sheaf}}

\newcommand{\RTop}{\categ{RTop}_{\infty}}
\newcommand{\LTop}{\categ{LTop}_{\infty}}
\newcommand{\LToppost}{\LTop^{\post}}

\newcommand{\Shapeprotrun}{\Pi_{<\infty}}

\newcommand{\cupperstarhyp}{c^{\ast,\hyp}}
\newcommand{\cupperstarpost}{c^{\ast,\post}}
\newcommand{\fupperstarpost}{f^{\ast,\post}}

\renewcommand{\D}{\operatorname{D}}
\renewcommand{\Dhat}{\widehat{\D}}

\newcommand{\cuppersharp}{c^{\sharp}}

\title{\Large Descent for sheaves on compact Hausdorff spaces}

\author{\normalsize Peter J. Haine}

\date{\normalsize \today \vspace{-2ex}} 

\begin{document}

\maketitle

\begin{abstract} 
	These notes explain some descent results for \categories of sheaves on compact Hausdorff spaces and derive some consequences.
	Specifically, given a compactly assembled \category $ \E $, we show that the functor sending a locally compact Hausdorff space $ X $ to the \category \smash{$ \Shpost(X;\E) $} of \textit{Postnikov complete} $ \E $-valued sheaves on $ X $ satisfies descent for proper surjections.
	This implies proper descent for \textit{left complete} derived \categories and that the functor \smash{$ \Shpost(-;\E) $} is a sheaf on the category of compact Hausdorff spaces equipped with the topology of finite jointly surjective families.
	Using this, we explain how to embed Postnikov complete sheaves on a locally compact Hausdorff space into condensed objects.
	This implies that the condensed and sheaf cohomologies of a locally compact Hausdorff space agree.
	\vspace{-1ex} 
\end{abstract}



\setcounter{tocdepth}{2}

\tableofcontents


\setcounter{section}{-1}

\section{Introduction}

The first goal of these notes is to explain some descent results for \categories of sheaves on locally compact Hausdorff spaces.
Our motivation comes from condensed/pyknotic mathematics developed by Clausen--Scholze \cites{Scholze:Condensedtalk}{Scholze:condensedtalknotes}{Scholze:analyticnotes}{Scholze:condensednotes}, in our joint work with Barwick \cites[Chapter 13]{arXiv:1807.03281}{arXiv:1904.09966}, and in Lurie's work on ultracategories \cites{Lurie:CatLogic}[Chapter 4]{Ultracategories}.
Write $ \Comp $ for the category of compact Hausdorff spaces.
The category $ \Comp $ has a Grothendieck topology where the covering families are finite families of jointly surjective maps.
Because of the simplicity of the Grothendieck topology, the sheaf condition is very explicit: a presheaf on $ \Comp $ is a sheaf if and only if it carries finite disjoint unions of compact Hausdorff spaces to finite products and satisfies descent for surjections. 

Our first goal is to answer the following question:

\begin{question}\label{quest:is_Sh_a_sheaf}
	Is the functor $ \Sh \colon \fromto{\Compop}{\Catinfty} $ that assigns a compact Hausdorff space $ K $ the \category $ \Sh(K) $ of sheaves of spaces on $ K $ a sheaf with respect to this topology?
\end{question}

\noindent Perhaps surprisingly, the answer to \Cref{quest:is_Sh_a_sheaf} is \textit{negative} (see \Cref{cor:Sh_Shhyp_not_sheaves}).
Moreover, if one replaces sheaves by hypersheaves, the answer to \Cref{quest:is_Sh_a_sheaf} is still negative.
The reason for this failure of descent is that every compact Hausdorff space admits a surjection from a profinite set, and the \category of sheaves on a profinite set satisfies a strong completeness property which the \category of (hyper)sheaves on a general compact Hausdorff space does not satisfy.
So it is not reasonable to ask for the \category of sheaves on a general compact Hausdorff spaces to be expressible as a limit of \categories satisfying this completeness property.


\subsection{Postnikov completion}

Since this completeness property is central to these notes, before stating the main results, let us briefly introduce it.
See \cref{subsec:Postnikov_completeness} for more details.

\begin{definition}
	Let $ X $ be a topological space.
	The \textit{Postnikov completion} of the \category of sheaves of spaces on on $ X $ is the inverse limit
	\begin{equation*}
		\Shpost(X) \colonequals \lim
		\left(
		\begin{tikzcd}
			\cdots \arrow[r, "\trun_{\leq n+1}"] & \Sh(X)_{\leq n+1} \arrow[r, "\trun_{\leq n}"] & \Sh(X)_{\leq n} \arrow[r, "\trun_{\leq n-1}"] & \cdots
		\end{tikzcd}
		\right)
	\end{equation*}
	of the \categories of sheaves of $ n $-truncated spaces along the truncation functors.
\end{definition}

Objects of $ \Shpost(X) $ are towers
\begin{equation*}
	\cdots \to F_{n+1} \to F_n \to \cdots \to F_0
\end{equation*}
where $ F_n $ is an sheaf of $ n $-truncated spaces on $ X $ such that $ \equivto{\trun_{\leq n} F_{n+1}}{F_n} $.
There is a natural left adjoint \smash{$ \fromto{\Sh(X)}{\Shpost(X)} $} sending a sheaf $ F $ to its Postnikov tower $ \{\trun_{\leq n} F\}_{\geq 0} $. 
We say that $ \Sh(X) $ is \textit{Postnikov complete} if this functor \smash{$ \fromto{\Sh(X)}{\Shpost(X)} $} is an equivalence.

\begin{example}
	The \topos of sheaves on a profinite set is Postnikov complete.
	On the other hand, the \topos of sheaves on the Hilbert cube $ \prod_{i \geq 1} [0,1] $ is not Postnikov complete.
\end{example}

For a presentable \category $ \E $, we write $ \Shpost(X;\E) $ for the tensor product $ \Shpost(X) \tensor \E $.
With stable coefficients this recovers the \textit{left-complete} derived \category of sheaves:

\begin{example}
	Let $ X $ be a topological space and let $ R $ be a ring.
	Write $ \D(X;R) $ for the derived \category of the abelian category of sheaves of $ R $-modules on $ X $.
	Then \smash{$ \Shpost(X;\D(R)) $} is the \textit{left completion} of $ \D(X;R) $ with respect to the standard \tstructure.%
	\footnote{We use \textit{homological} indexing for our \tstructures.}
	That is, \smash{$ \Shpost(X;\D(R)) $} is the limit of the diagram of \categories 
	\begin{equation*}
		\begin{tikzcd}
			\cdots \arrow[r, "\trun_{\leq n+1}"] & \D(X;R)_{\leq n+1} \arrow[r, "\trun_{\leq n}"] & \D(X;R)_{\leq n} \arrow[r, "\trun_{\leq n-1}"] & \cdots 
		\end{tikzcd}
	\end{equation*}
	along the truncation functors with respect to the standard \tstructure.
\end{example}


\subsection{Descent for Postnikov complete sheaves}

The following is the main descent result of these notes.
Note that all compactly generated \categories are compactly assembled (see \Cref{rec:compactly_assembled}).

\begin{theorem}[(\Cref{cor:general_descent})]\label{intro_thm:general_descent}
	Let $ \E $ be a compactly assembled \category.
	Then for every proper surjection of locally compact Hausdorff spaces $ p \colon \surjto{X}{Y} $, natural functor
	\begin{equation*}
		\begin{tikzcd}[sep=1.5em]
		    \Shpost(Y;\E) \arrow[r] & \lim\Bigg(\Shpost(X;\E) \arrow[r, shift left=0.75ex, "\prupperstar_1"] \arrow[r, shift right=0.75ex, "\prupperstar_2"'] & \Shpost(X \cross_Y X; \E) \arrow[l] \arrow[r] \arrow[r, shift left=1.5ex] \arrow[r, shift right=1.5ex] &  \cdots \Bigg) \arrow[l, shift left=0.75ex] \arrow[l, shift right=0.75ex] 
		\end{tikzcd}
	\end{equation*}
	is an equivalence in $ \Catinfty $.
	Consequently, the functor \smash{$ \Shpost(-;\E) \colon \fromto{\Compop}{\Catinfty} $} is a hypersheaf of \categories on the site of compact Hausdorff spaces.
\end{theorem}

\begin{example}
	Let $ R $ be a ring.
	Then the functor $ \Dhat(-;R)\colon \fromto{\Compop}{\Catinfty} $ carrying a compact Hausdorff space to its left complete derived \category is a hypersheaf.
	Hence the functor 
	\begin{equation*}
		\D(-;R)_{<\infty} \colon \fromto{\Compop}{\Catinfty}
	\end{equation*}
	that sends a compact Hausdorff space $ K $ to its bounded-above derived \category%
	\footnote{What we write as $ \D(K;R)_{<\infty} $ is often written as $ \D^+(K;R) $.}
	is also a hypersheaf of \categories.
\end{example}

Passing to global sections shows that sheaf cohomology also satisfies proper descent.

\begin{corollary}
	Let $ R $ be a connective $ \Eup_1 $-ring spectrum and $ M $ a bounded-above left $ R $-module spectrum. 
	The functor
	\begin{equation*}
		\RGammasheaf(-;M) \colon \fromto{\Compop}{\LMod(R)}
	\end{equation*}
	is a hypersheaf.
\end{corollary}

\noindent Note that if $ R $ is an ordinary ring, then the \category $ \LMod(R) $ is the derived \category $ \D(R) $.


\subsection{The comparison between sheaf and condensed cohomology}

Part of our motivation for proving \Cref{intro_thm:general_descent} is that it has a number of consequences.
One application is a generalization of work of Dyckhoff and Clausen--Scholze that compares sheaf cohomology with condensed cohomology.
Let $ X $ be a locally compact Hausdorff space.
We can also regard $ X $ as an object of the \category $ \Sh(\Comp) $ via the restricted Yoneda embedding.
Dyckhoff \cites[Theorem 3.11]{MR0448318}{MR394633} and Clausen--Scholze \cite[Theorem 3.2]{Scholze:condensednotes} showed that if $ A $ is an abelian group, and $ X $ is \textit{compact} then there is an isomorphism
\begin{equation*}
	\Hsheaf^{\ast}(X;A) \isomorphism \Hcond^{\ast}(X;A)
\end{equation*}
from the sheaf cohomology of $ X $ to the cohomology of $ X $ regarded as an object $ \Sh(\Comp) $.

We extend this result in two directions: to locally compact Hausdorff spaces and to very general coefficients.
The comparison map between sheaf and condensed cohomology is induced by a natural geometric morphism
\begin{equation*}
	c_{X,\ast} \colon \fromto{\Sh(\Comp)_{/X}}{\Sh(X)}
\end{equation*}
given by sending a sheaf $ G \colon \fromto{\Compop}{\Spc} $ to the sheaf on $ X $ defined by
\begin{equation*}
	c_{X,\ast}(G)(U) \colonequals \Map_{\Sh(\Comp)_{/X}}(U,G) \period
\end{equation*} 
(See \cref{subsec:the_comparison_functor} for details.)

Since cohomology is computed by derived global sections, to show that the sheaf and condensed cohomologies of $ X $ agree, it suffices to show that $ \cupperstar_{X} $ is fully faithful.
Again, this is generally only true \textit{after} Postnikov completion (see \Cref{warning:cupperstar_not_fully_faithful,warning:cupperstarhyp_not_fully_faithful}).

\begin{proposition}[(\Cref{cor:cupperstarpost_fully_faithful_LCH})]\label{intro_prop:cupperstarpost_fully_faithful_LCH}
	Let $ X $ be a locally compact Hausdorff space and let $ \E $ be a compactly assembled \category.
	Then the pullback functor
	\begin{equation*}
		\cupperstarpost \colon \fromto{\Shpost(X;\E)}{\Shpost(\CompX;\E)}
	\end{equation*}
	is fully faithful.
\end{proposition}

\Cref{intro_prop:cupperstarpost_fully_faithful_LCH} implies, for example:

\begin{corollary}[(\Cref{cor:cohomology_comparison_general})]\label{intro_cor:cohomology_comparison_general}
	Let $ X $ be locally compact Hausdorff space.
	Let $ R $ be a connective $ \Eup_1 $-ring spectrum and let $ M $ be a bounded-above left $ R $-module spectrum. 
	Then the natural map
	\begin{equation*}
		\fromto{\RGammasheaf(X;M)}{\RGammacond(X;M)}
	\end{equation*}
	is an equivalence in the \category $ \LMod(R) $ of left $ R $-module spectra.
\end{corollary}

\noindent As a consequence, the condensed, singular, and sheaf cohomologies of a topological space admitting a locally finite CW structure all agree (see \Cref{rmk:comparing_condensed_singular_sheaf,nul:CW_comparing_condensed_singular_sheaf}).


\subsection{Linear overview}\label{subsec:linear_overview}

We imagine that the reader might be interested in condensed/pyknotic mathematics but not necessarily familiar with all of the intricacies about the theory of \topoi.
With this in mind, in \cref{sec:background}, we review the basics of hypercomplete and Postnikov complete \topoi; the familiar reader can safely skip this section.
\Cref{sec:proper_descent} proves \Cref{intro_thm:general_descent} and derives some consequences in shape theory.
In \cref{sec:cohomology_comparison}, we construct the comparison geometric morphism 
\begin{equation*}
	\clowerstar \colon \fromto{\Sh(\CompL)}{\Sh(L)} 
\end{equation*}
and record its basic properties.
\Cref{sec:full_faithfulness_of_comparison} is dedicated to proving \Cref{intro_prop:cupperstarpost_fully_faithful_LCH}.


\begin{acknowledgments}
	We thank Ko Aoki, Clark Barwick, Marc Hoyois, Jacob Lurie, Mark Macerato, Zhouhang Mao, Denis Nardin, Piotr Pstrągowski, Marco Volpe, Sebastian Wolf, and Tong Zhou for helpful comments and conversations around the contents of these notes.
	Special thanks are due to Marc Hoyois and Jacob Lurie for explaining \Cref{ex:not_all_hypersheaves_are_Postnikov_complete} to us.
	These notes are clearly highly influenced by Dustin Clausen and Peter Scholze's ideas; we would like to thank them too.

	We gratefully acknowledge support from the UC President's Postdoctoral Fellowship and NSF Mathematical Sciences Postdoctoral Research Fellowship under Grant \#DMS-2102957. 
\end{acknowledgments}


\section{Background}\label{sec:background}

Recall the following fundamental results about the \category of spaces.
\begin{enumerate}
	\item \textit{Whitehead's Theorem:} A map $ f \colon \fromto{X}{Y} $ of spaces is an equivalence if and only if $ f $ induces a bijection on connected components and isomorphisms on homotopy groups at each basepoint.
	Said differently, $ f $ is an equivalence if and only if for each $ n \geq 0 $, the induced map on $ n $-truncations $ \trun_{\leq n}(f) \colon \fromto{\trun_{\leq n}(X)}{\trun_{\leq n}(Y)} $ is an equivalence.

	\item \textit{Convergence of Postnikov towers:} Every space $ X $ is the limit of its Postnikov tower.
	That is, the the natural map $ \fromto{X}{\lim_{n \geq 0} \trun_{\leq n}(X)} $ is an equivalence.
\end{enumerate}
The statement of Whitehead's Theorem and the convergence of Postnikov towers make can be formulated in an arbitrary \topos.
However, even for the \topos of sheaves on a compact Hausdorff space, neither result need hold (see \Cref{ex:not_all_hypersheaves_are_Postnikov_complete}).
The purpose of this section is to review two completion procedures (\textit{hypercompletion} and \textit{Postnikov completion}) that force Whitehead's Theorem to hold and Postnikov towers to converge, respectively.
In the higher-categorical world, these give rise to three natural `sheaf theories' (sheaves, hypersheaves, and Postnikov complete sheaves) extending the classical theory of sheaves on a topological space.
They all have the same truncated objects, so the subtle differences between these theories only appears when considering `unbounded' objects.

\Cref{subsec:hypercompleteness} reviews the basics of hypercompleteness; in the process, we set some notation.
In \cref{subsec:Postnikov_completeness} we review Postnikov completeness.
In \cref{subsec:condensed_pyknotic}, we recall the basic setup of condensed/pyknotic mathematics.


\subsection{Hypercompleteness}\label{subsec:hypercompleteness}

In this subsection, we set up some notation and review the basics of hypercompletions of \topoi.
We refer the reader unfamiliar with hypercomplete objects and hypercompletion to \cite[\S\S \HTTsubseclink{6.5.2}--\HTTsubseclink{6.5.4}]{HTT}, \cite[\S 3.11]{arXiv:1807.03281}, or \cite[\S1.2]{arXiv:2010.06473} for further reading on the subject.

\begin{notation}
	Write $ \Spc $ for the \category of spaces and $ \Catinfty $ for the \category of \categories.
\end{notation}

\begin{notation}
	Let $ \Ccal $ be \asite and $ \E $ a presentable \category.
	We write
	\begin{equation*} 
		\PSh(\Ccal;\E) \colonequals \Fun(\Ccal^{\op}, \E) 
	\end{equation*}
	for the \category of $\E$-valued presheaves on $ \Ccal $.
	We write $\Sh(\Ccal;\E) \subset \PSh(\Ccal;\E) $ for the full subcategory spanned by $\E$-valued sheaves.
	When $\E = \Spc$, we simply write
	\begin{equation*}
		\PSh(\Ccal) \colonequals \PSh(\Ccal;\Spc) \andeq \Sh(\Ccal) \colonequals \Sh(\Ccal;\Spc) \period
	\end{equation*}
\end{notation}	

\begin{nul}
	The \categories $\PSh(\Ccal;\E)$ and $\Sh(\Ccal;\E)$ are naturally identified with the tensor products of presentable \categories $\PSh(\Ccal) \tensor \E $ and $\Sh(\Ccal) \tensor \E $ \cite[\SAGthm{Remark}{1.3.1.6} \& \SAGthm{Proposition}{1.3.1.7}]{SAG}.
\end{nul}

\begin{notation}\label{nul:topological_space}
	Let $ X $ be a topological space.
	We write $ \Open(X) $ the poset of open subsets of $ X $, ordered by inclusion.
	We regard $ \Open(X) $ as a site with the covering families given by open covers.
	We write
	\begin{equation*}
		\PSh(X;\E) \colonequals \PSh(\Open(X);\E) \andeq \Sh(X;\E) \colonequals \Sh(\Open(X);\E) \period
	\end{equation*}
\end{notation}

\begin{notation}
	Let $ \E $ be a presentable \category and $ \flowerstar \colon \fromto{\X}{\Y} $ a geometric morphism of \topoi.
	For simplicity, we also denote the tensor product $ \flowerstar \tensor \E \colon \fromto{\X \tensor \E}{\Y \tensor \E} $ by $ \flowerstar $.
\end{notation}

\begin{recollection}[(hypercompleteness)]
	Let $ \X $ be \atopos.
	The \category of \defn{hypercomplete} objects of $ \X $ is the full subcategory $ \Xhyp \subset \X $ spanned by those objects that are local with respect to the $ \infty $-connected morphisms.
	The inclusion $ \Xhyp \subset \X $ admits a left exact left adjoint; hence $ \Xhyp $ is also \atopos.

	As the name suggests, $ \Xhyp $ can also be identified as the full subcategory of $ \X $ spanned by those objects that satisfy descent for hypercovers; see \cites[\HTTthm{Corollary}{6.5.3.13}]{HTT}[\S1]{MR2034012}[Corollary 3.4.7]{MR2137288}.
\end{recollection}

\begin{notation}
	Let $ \X $ be \atopos and $ n \geq 0 $ an integer.
	Write $ \X_{\leq n} \subset \X  $ for the full subcategory spanned by the $ n $-truncated objects, and write $ \trun_{\leq n} \colon \fromto{\X}{\X_{\leq n}} $ for the left adjoint to the inclusion. 
\end{notation}

\begin{remark}\label{rmk:hypersheaves_in_an_n-category}
	For each integer $ n \geq 0 $, the inclusion $ \incto{\Xhyp}{\X} $ restricts to an equivalence $ \equivto{(\Xhyp)_{\leq n}}{\X_{\leq n}} $ on subcategories of $ n $-truncated objects \HTT{Lemma}{6.5.2.9}.
	As a consequence, given an integer $ n \geq 0 $ and presentable $ n $-category $ \E $, we have $ \Xhyp \tensor \E \equivalent \X \tensor \E $ \HA{Example}{4.8.1.22}.
	In particular, every sheaf of sets is hypercomplete.
\end{remark}

\begin{notation}
	Let $ \Ccal $ be a site and $ \Ecal $ a presentable \category.
	We write $ \Shhyp(\Ccal) \colonequals \Sh(\Ccal)^{\hyp} $ and 
	\begin{equation*}
		\Shhyp(\Ccal;\Ecal) \colonequals \Shhyp(\Ccal) \tensor \Ecal \period
	\end{equation*}
	We refer to objects of $ \Shhyp(\Ccal;\E) $ as $ \E $-valued \defn{hypersheaves}.
	We use analogous notation for hypersheaves on a topological space.
\end{notation}

For reasonable coefficients, hypersheaves on a topological space can be identified very explicitly.

\begin{recollection}[(compactly assembled \categories)]\label{rec:compactly_assembled}
	A presentable \category $ \E $ is \textit{compactly assembled} if $ \E $ is a retract of a compactly generated \category regarded as an object of the \category \smash{$ \PrL $} of presentable \categories and left adjoints \cite[\SAGthm{Definition}{21.1.2.1} \& \SAGthm{Theorem}{21.1.2.18}]{SAG}.
	A \textit{stable} presentable \category $ \E $ is compactly assembled if and only if $ \E $ is dualizable in the symmetric monoidal \category $ \PrLst $ of stable presentable \categories and left adjoints \SAG{Proposition}{D.7.3.1}.

\end{recollection}

\begin{remark}\label{nul:Shhyp_via_stalks}
	Let $ X $ be a topological space and $ \E $ a compactly assembled \category.
	Then the subcategory
	\begin{equation*}
		\Shhyp(X;\E) \subset \Sh(X;\E)
	\end{equation*}
	is the localization obtained by inverting all morphisms that induce equivalences on stalks \cites[\HAappthm{Lemma}{A.3.9}]{HA}[Lemma 2.11]{arXiv:2108.03545}.
\end{remark}

\begin{notation}
	Write $ \LTop $ for the \category with objects \topoi and morphisms left exact left adjoints.
\end{notation}

\begin{nul}\label{nul:limits_in_LTop_computed_in_Cat}
	We repeatedly use the fact that the forgetful functor $ \fromto{\LTop}{\Catinfty} $ preserves limits \HTT{Proposition}{6.3.2.3}.
\end{nul}

\begin{nul}
	Hypercompletion defines a functor $ (-)^{\hyp} \colon \fromto{\LTop}{\LTop} $ left adjoint to the inclusion of hypercomplete \topoi into $ \LTop $.
\end{nul}


\subsection{Postnikov completeness}\label{subsec:Postnikov_completeness}

In this subsection, we review the basics of \textit{Postnikov completions} of \topoi.
We refer the unfamiliar reader to \cites[\HTTsubsec{5.5.6}]{HTT}[\S\S\SAGsubseclink{A.7.2} \& \SAGsubseclink{A.7.2}]{SAG}[\S3.2]{arXiv:1807.03281} for more background.
The \category of spaces actually satisfies a stronger property than the requirement that every object be the limit of its Postnikov tower: the entire \category can be recovered as the limit of the subcategories of $ n $-truncated spaces along the truncation functors.
This is the property that we want to generalize to arbitrary \topoi.

\begin{definition}\label{def:Postnikovstuff}
	Let $ \X $ be \atopos.
	The \defn{Postnikov completion} of $ \X $ is the limit
	\begin{equation*}
		\Xpost \colonequals \lim \bigg(
		\begin{tikzcd}[sep=1.5em]
			\cdots \arrow[r] & \X_{\leq n+1} \arrow[r, "\trun_{\leq n}"] & \X_{\leq n} \arrow[r] & \cdots \arrow[r, "\trun_{\leq 0}"] & \X_{\leq 0} 
		\end{tikzcd}\bigg)
	\end{equation*}
	formed in $ \Catinfty $.
	Thus objects of $ \Xpost $ are given by towers
	\begin{equation*}
		\cdots \to V_{n+1} \to V_n \to \cdots \to V_0
	\end{equation*}
	in $ \X $, where $ V_n $ is $ n $-truncated and the map $ V_{n+1} \to V_n $ exhibits $ V_n $ as the $ n $-truncation of $ V_{n+1} $.
\end{definition}

\begin{nul}\label{nul:description_of_Postnikov_completion_adjunction}
	The Posnikov completion $ \Xpost $ is also \atopos.
	Moreover, there is a natural left exact left adjoint $ \tupperstar \colon \fromto{\X}{\Xpost} $ defined by sending an object to its Postnikov tower:
	\begin{equation*}
		\tupperstar(U) \colonequals \{\trun_{\leq n} U\}_{n \geq 0} \period
	\end{equation*}
	See \SAG{Theorem}{A.7.2.4}.
	The right adjoint $ \tlowerstar \colon \fromto{\Xpost}{\X} $ sends a tower $ \{V_{n}\}_{n \geq 0} $ to the limit $ \lim_{n \geq 0} V_{n} $ formed in $ \X $ \SAG{Remark}{A.7.3.6}.
\end{nul}

\begin{observation}\label{obs:pullback_to_Postnikov_completion_fully_faithful}
	The functor $ \tupperstar \colon \fromto{\X}{\Xpost} $ is fully faithful if and only if for each object $ U \in \X $, the natural map $ \fromto{U}{\lim_{n \geq 0} \trun_{\leq n} U} $ is an equivalence.
	That is, $ \tupperstar $ if fully faithful if and only if every object of $ \X $ is the limit of its Postnikov tower.
\end{observation}

\begin{remark}\label{nul:Postnikov_completion_equivalence_on_truncated}
	For each integer $ n \geq 0 $, the functor $ \tupperstar \colon \fromto{\X}{\Xpost} $ restricts to an equivalence $ \equivto{\X_{\leq n}}{(\Xpost)_{\leq n}} $ \SAG{Corollary}{A.7.3.8}.
	Moreover, given an integer $ n \geq 0 $ and presentable $ n $-category $ \E $, we have
	\begin{equation*}
		\Xpost \tensor \E \equivalent \Xhyp \tensor \E \equivalent \X \tensor \E \period
	\end{equation*} 
\end{remark}

\begin{warning}
	The right adjoint $ \tlowerstar \colon \fromto{\Xpost}{\X} $ is fully faithful if and only if $ \Xhyp = \Xpost $.
\end{warning}

\begin{definition}
	We say that \atopos $ \X $ is \defn{Postnikov complete} if the functor $ \tupperstar \colon \fromto{\X}{\Xpost} $ is an equivalence of \categories.
\end{definition}

\begin{remark}
	The natural geometric morphism $ \incto{\Xhyp}{\X} $ induces an equivalence 
	\begin{equation*}
		\equivto{(\Xhyp)^{\post}}{\Xpost} \period
	\end{equation*}
	Moreover, if $ \X $ is Postnikov complete, then $ \X $ is hypercomplete.
	However, the converse is false (see \Cref{ex:not_all_hypersheaves_are_Postnikov_complete}).
\end{remark}

\begin{observation}\label{obs:Postnikov_completeness_is_convergence}
	Let $ \X $ be \atopos.
	Then $ \X $ is Postnikov complete if and only if the pushforward $ \tlowerstar \colon \fromto{\Xpost}{\X} $ is conservative and the pullback $ \tupperstar $ is fully faithful.
	From the explicit descriptions of $ \tupperstar $ and $ \tlowerstar $, we see that $ \X $ is Postnikov complete if and only if the following conditions are satisfied:
	\begin{enumerate}[label=\stlabel{obs:Postnikov_completeness_is_convergence}]
		\item For each $ U \in \X $, the natural map $ \fromto{U}{\lim_{n \geq 0} \trun_{\leq n} U} $ is an equivalence.
		
		\item For each integer $ n \geq 0 $, the functor $ \tlowerstar \colon \fromto{\Xpost}{\X} $ commutes with $ n $-truncation.
	\end{enumerate}
	See also \HTT{Proposition}{5.5.6.26}.
\end{observation}

\begin{warning}
	As far as we are aware, it is not known if there exists \atopos $ \X $ such that every object of $ \X $ is the limit of its Postnikov tower, but $ \X $ is not Postnikov complete.  
\end{warning}

\begin{notation}
	Write $ \LToppost \subset \LTop $ for the full subcategory spanned by the Postnikov complete \topoi.
\end{notation}

\begin{nul}[{}]\label{rec:Postnikov_completion_preserves_limits_colimits}
	Postnikov completion defines a functor
	\begin{equation*}
		(-)^{\post} \colon \fromto{\LTop}{\LToppost}
	\end{equation*}
	which is left adjoint to the inclusion \SAG{Corollary}{A.7.2.6}.
	The functor $ (-)^{\post} $ is also a right adjoint \SAG{Corollary}{A.7.2.7}.
	Hence the full subcategory $ \LToppost \subset \LTop $ is closed under limits.
	As a consequence of \cref{nul:limits_in_LTop_computed_in_Cat}, the forgetful functor \smash{$ \fromto{\LToppost}{\Catinfty} $} preserves limits.
\end{nul}

To prove \Cref{intro_thm:general_descent}, use the following reformulation of what it means for a diagram of Postnikov complete \topoi to be a limit diagram:

\begin{lemma}\label{lem:limits_and_Postnikov_completion}
	Let $ \Ical $ be \acategory and $ \X_{\bullet} \colon \fromto{\Ical^{\smalltriangleleft}}{\LTop} $ a diagram of \topoi.
	The following are equivalent:
	\begin{enumerate}[label=\stlabel{lem:limits_and_Postnikov_completion}, ref=\arabic*]
		\item\label{lem:limits_and_Postnikov_completion.1} For each integer $ n \geq 0 $, the diagram $ (\X_{\bullet})_{\leq n} \colon \fromto{\Ical^{\smalltriangleleft}}{\Catinfty} $ is a limit diagram.

		\item\label{lem:limits_and_Postnikov_completion.2} The diagram of Postnikov complete \topoi $ \X_{\bullet}^{\post} \colon \fromto{\Ical^{\smalltriangleleft}}{\Catinfty} $ is a limit diagram.
	\end{enumerate}
\end{lemma}

\begin{proof}
	By the definition of Postnikov completion and the fact that limits commute, we see that \enumref{lem:limits_and_Postnikov_completion}{1} $ \Rightarrow $ \enumref{lem:limits_and_Postnikov_completion}{2}.
	To prove that \enumref{lem:limits_and_Postnikov_completion}{2} $ \Rightarrow $ \enumref{lem:limits_and_Postnikov_completion}{1}, tensor the limit diagram $ \X_{\bullet}^{\post} \colon \fromto{\Ical^{\smalltriangleleft}}{\Catinfty} $ with $ \Spc_{\leq n} $ and apply \cite[Lemma 2.15]{arXiv:2108.03545}.
\end{proof}

We finish this subsection with some examples of Postnikov complete \topoi from topology.

\begin{notation}
	Let $ \Ccal $ be \asite and $ \E $ a presentable \category.
	Write $ \Shpost(\Ccal) \colonequals \Sh(\Ccal)^{\post} $ and 
	\begin{equation*}
		\Shpost(\Ccal;\E) \colonequals \Shpost(\Ccal) \tensor \E \period
	\end{equation*}
	We refer to objects of $ \Shpost(\Ccal;\E) $ as \defn{Postnikov complete sheaves} on $ \Ccal $.%
	\footnote{Since $ \Shpost(\Ccal;\E) $ is not generally a subcategory of $ \Sh(\Ccal;\E) $, this is slightly abusive.}
	We use analogous notation for Postnikov complete sheaves on a topological space.
\end{notation}

\begin{example}\label{ex:Sh_on_profinset}
	Let $ S $ be a profinite set.
	Since $ S $ is $ 0 $-dimensional, by \cite[\HTTthm{Corollary}{7.2.1.10}, \HTTthm{Theorem}{7.2.3.6}, \& \HTTthm{Remark}{7.2.4.18}]{HTT} the \topos $ \Sh(S) $ is Postnikov complete.
\end{example}

\begin{example}[{\cite{MO:168526}}]\label{ex:Sh_on_CW}
	Let $ X $ be a topological space.
	If $ X $ admits a CW structure, then the \topos $ \Sh(X) $ is Postnikov complete.
\end{example}

For hypersheaves and Postnikov complete sheaves, pullbacks along surjections are conservative:

\begin{observation}\label{obs:pullback_conservativity}
	Let $ p \colon \surjto{X}{Y} $ be a surjection of topological spaces.
	Since \smash{$ \Shhyp(Y) $} has enough points and the points of \smash{$ \Shhyp(Y) $} are in natural bijection with the underlying set of $ Y $, the pullback functor 
	\begin{equation*}
		p^{\ast,\hyp} \colon \fromto{\Shhyp(Y)}{\Shhyp(X)}
	\end{equation*}
	is conservative.
	In particular, the functor $ \pupperstar \colon \fromto{\Sh(Y)}{\Sh(X)} $ is conservative when restricted to the subcategory of truncated objects.
	Hence $ p^{\ast,\post} \colon \fromto{\Shpost(Y)}{\Shpost(X)} $ is also conservative.
\end{observation}


\subsection{Condensed/pyknotic mathematics}\label{subsec:condensed_pyknotic}

In this subsection, we briefly recall the formalism of condensed/pyknotic mathematics.
We refer the reader to \cites[\S13.3]{arXiv:1807.03281}{arXiv:1904.09966}{Scholze:complexnotes}{Scholze:Condensedtalk}{Scholze:condensedtalknotes}[Lecture I]{Scholze:analyticnotes}{Scholze:condensednotes} for more details and motivation.

\begin{notation}	
	Write $ \Top $ for the category of topological spaces and  $ \Comp \subset \Top $ for the full subcategory spanned by the compact Hausdorff spaces.
	We regard $ \Comp $ as a site where the covering families are finite families of jointly surjective maps.
\end{notation}

\begin{remark}[(set theory)]
	Since $ \Comp $ is a large category, one has to be careful about talking about sheaves on $ \Comp $.
	To do this, we adopt the set-theoretic conventions of \cites[\S13.3]{arXiv:1807.03281}{arXiv:1904.09966}; this uses universes to deal with the set theory.
	Clausen and Scholze \cite{Scholze:condensednotes} use alternative set-theoretic foundations that avoid using universes.
	This minor difference is only in the set theory and does not affect any arguments in an essential way, so we will not mention it again.
\end{remark}

In this setting, the sheaf condition is particularly easy to formulate:

\begin{observation}\label{obs:sheaf_condition}
	Let $ \Dcal $ be \acategory.
	A presheaf $ F \colon \fromto{\Compop}{\Dcal} $ is a sheaf if and only if the following conditions are satisfied:
	\begin{enumerate}[label=\stlabel{obs:sheaf_condition}]
		\item The functor $ F $ preserves finite products.
		That is, $ F $ carries finite coproducts of compact Hausdorff spaces to finite products in $ \Dcal $.

		\item For every surjection of compact Hausdorff spaces $ p \colon \surjto{X}{Y} $,
		the augmented cosimplicial diagram
		\begin{equation*}
			\begin{tikzcd}[sep=1.5em]
			    F(Y) \arrow[r, "\pupperstar"] & F(X) \arrow[r, shift left=0.75ex] \arrow[r, shift right=0.75ex] & F(X \cross_Y X) \arrow[l] \arrow[r] \arrow[r, shift left=1.5ex] \arrow[r, shift right=1.5ex] &  \cdots  \arrow[l, shift left=0.75ex] \arrow[l, shift right=0.75ex] 
			\end{tikzcd}
		\end{equation*}
		obtained by applying $ F $ to the Čech nerve of $ p $ exhibits $ F(Y) $ as the limit of its restriction to $ \DDelta \subset \Deltaplus $.
	\end{enumerate}
	See \SAG{Proposition}{A.3.3.1}.
\end{observation}

\begin{nul}
	Every representable presheaf on $ \Comp $ is a sheaf.
	Moreover, the topology on $ \Comp $ is designed exactly so that the Yoneda embedding $ \incto{\Comp}{\Sh(\Comp)} $ preserves finite coproducts and carries surjections to effective epimorphisms.
\end{nul}

\begin{notation}
	Write $ \Extr \subset \Comp $ for the full subcategory spanned by the \textit{extremally disconnected} profinite sets.
	The extremally disconnected profinte sets are exactly the projective objects of the category $ \Comp $ \cites{MR0121775}[Chapter III, \S3.7]{MR861951}.
\end{notation}

\begin{nul}\label{nul:extremally_disconnected_basis}
	Let $ K $ be a compact Hausdorff space.
	Write $ K^{\disc} $ for the underlying set of $ K $ equipped with the discrete topology.
	There is a natural surjection $ \surjto{\SC(K^{\disc})}{K} $ from the Stone--Čech compactification of $ K^{\disc} $ to $ K $.
	Since the profinite set $ \SC(K^{\disc}) $ is extremally disconnected, the subcategory $ \Extr \subset \Comp $ is a basis for the Grothendieck topology on $ \Comp $.
	Therefore, restriction of presheaves defines an equivalence of \categories
	\begin{equation*}
		\Shhyp(\Comp) \equivalence \Shhyp(\Extr)
	\end{equation*}
	with inverse given by right Kan extension \cites[Corollary A.8]{arXiv:2001.00319}[Corollary 3.12.14]{arXiv:1807.03281}.
\end{nul}

\begin{nul}\label{nul:ShhypComp_Postnikov_complete}
	Since every surjection of extremally disconnected profinite sets admits a section, a presheaf $ F $ on $ \Extr $ is a sheaf if and only if $ F $ preserves finite products.
	Moreover, by \cite[Lemma 2.4.10]{arXiv:1904.09966}, the \topos $ \Sh(\Extr) $ is Postnikov complete.
	Hence \smash{$ \Shhyp(\Comp) $} is also Postnikov complete.
\end{nul}

Most reasonable topological spaces embed into sheaves on $ \Comp $:

\begin{notation}
	Write $ \yo \colon \fromto{\Top}{\Sh(\Comp)} $ for the restricted Yoneda functor defined by
	\begin{equation*}
		\yo(X)(K) \colonequals \Map_{\Top}(K,X) \period
	\end{equation*}
	When it does not cause confusion, we also simply denote $ \yo(X) \in \Sh(\Comp) $ by $ X $.
\end{notation}

\begin{nul}
	The functor $ \yo $ is not fully faithful.
	However, $ \yo $ is fully faithful when restricted to the the full subcategory of \defn{compactly generated} topological spaces.
\end{nul}

\begin{nul}
	Note that the functor $ \yo \colon \fromto{\Top}{\Sh(\Comp)} $ preserves limits.
	The functor $ \yo $ does not preserve arbitrary colimits.
	However, in \cref{subsec:descent_for_open_covers} we show that $ \yo $ behaves well with open covers, coproducts, and proper surjections.
\end{nul}


\section{Proper descent}\label{sec:proper_descent}

Let $ \E $ be a compactly assembled \category (see \Cref{rec:compactly_assembled}).
In this section, we show that the functor sending a locally compact Hausdorff space $ X $ to the \category $ \Shpost(X;\E) $ of Postnikov complete sheaves on $ X $ satisfies descent for proper surjections in the following sense.

\begin{notation}
	Write $ \LCH \subset \Top $ for the full subcategory spanned by the locally compact Hausdorff spaces.
\end{notation}

\begin{definition}
	Let $ \Dcal $ be \acategory.
	We say that a functor $ F \colon \fromto{\LCHop}{\Dcal} $ \defn{satisfies proper descent} if for every proper surjection of locally compact Hausdorff spaces $ p \colon \surjto{X}{Y} $, the augmented cosimplicial diagram
	\begin{equation*}
		\begin{tikzcd}[sep=1.5em]
		    F(Y) \arrow[r, "\pupperstar"] & F(X) \arrow[r, shift left=0.75ex] \arrow[r, shift right=0.75ex] & F(X \cross_Y X) \arrow[l] \arrow[r] \arrow[r, shift left=1.5ex] \arrow[r, shift right=1.5ex] &  \cdots  \arrow[l, shift left=0.75ex] \arrow[l, shift right=0.75ex] 
		\end{tikzcd}
	\end{equation*}
	obtained by applying $ F $ to the Čech nerve of $ p $ exhibits $ F(Y) $ as the limit of its restriction to $ \DDelta \subset \Deltaplus $.
\end{definition}

\Cref{subsec:proper_descent} proves \Cref{intro_thm:general_descent} (see \Cref{cor:general_descent}).
In \cref{subsec:shape_theory} we record some shape-theoretic consequences.
In \cref{subsec:necesity_of_Postnikov_completion}, we explain why \Cref{intro_thm:general_descent} does not hold before Postnikov completion. 


\subsection{Proper descent for Postnikov sheaves}\label{subsec:proper_descent}

To prove that the functor $ \goesto{X}{\Shpost(X)} $ satisfies proper descent we apply the following criterion:

\begin{proposition}[\HA{Corollary}{4.7.5.3}]\label{prop:HA.4.7.5.3}
	Let $ \Scal^{\bullet} \colon \fromto{\Deltaplus}{\Catinfty} $ be an augmented cosimplicial \category.
	Let $ G \colon \fromto{\Scal^{-1}}{\Scal^0} $ denote the agumentation.
	Assume that:
	\begin{enumerate}[label=\stlabel{prop:HA.4.7.5.3}, ref=\arabic*]
		\item\label{prop:HA.4.7.5.3.1} The \category $ \Scal^{-1} $ admits totalizations of $ G $-split cosimplicial objects, and those totalizations are preserved by $ G $.

		\item\label{prop:HA.4.7.5.3.2} For every morphism $ \alpha \colon \fromto{[m]}{[n]} $ in $ \Deltaplus $, the square
		\begin{equation*}
			\begin{tikzcd}
				\Scal^m \arrow[r, "d^0"] \arrow[d, "\alphaupperstar"'] & \Scal^{m+1} \arrow[d, "{([0] \star \alpha)\upperstar}"] \\
				\Scal^n \arrow[r, "d^0"'] & \Scal^{n+1}
			\end{tikzcd}
		\end{equation*}
		is horizontally right adjointable.
		(Note that, in particular, this requires that the coface functors $ d^0 $ be left adjoints.)
	\end{enumerate}
	Then the natural functor \smash{$ \theta \colon \fromto{\Scal^{-1}}{\lim_{[n] \in \DDelta} \Scal^n} $} admits a fully faithful right adjoint.
	Moreover, if $ G $ is conservative, then $ \theta $ is an equivalence.
\end{proposition}

We are interested in applying \Cref{prop:HA.4.7.5.3} in the case where $ \Scal^{\bullet} $ is obtained by applying sheaves to the Čech nerve of a proper surjection.


\begin{example}\label{ex:descent_hypothesis_2}
	Let $ \E $ be a presentable \category which is compactly generated or stable, and let $ p \colon \surjto{X}{Y} $ be a proper surjection of locally compact Hausdorff spaces.
	The Proper Basechange Theorem \cites[\HTTthm{Corollary}{7.3.1.18}]{HTT}[Subexample 3.15]{arXiv:2108.03545} implies that the augmented cosimplicial diagram
	\begin{equation}\label{diagam:Sh_of_Cech_nerve}
		\begin{tikzcd}[sep=1.5em]
		    \Sh(Y;\E) \arrow[r, "\pupperstar"] & \Sh(X;\E) \arrow[r, shift left=0.75ex] \arrow[r, shift right=0.75ex] & \Sh(X \cross_Y X;\E) \arrow[l] \arrow[r] \arrow[r, shift left=1.5ex] \arrow[r, shift right=1.5ex] &  \cdots  \arrow[l, shift left=0.75ex] \arrow[l, shift right=0.75ex] 
		\end{tikzcd}
	\end{equation}
	satisfies hypothesis \enumref{prop:HA.4.7.5.3}{2}.
\end{example}

\begin{observation}\label{ex:descent_hypothesis_1}
	In the setting of \Cref{ex:descent_hypothesis_2}, the functor $ \pupperstar \colon \fromto{\Sh(Y;\E)}{\Sh(X;\E)} $ is left exact.
	If there is an integer $ n \geq 0 $ such that $ \E $ is an $ n $-category, then the totalizations in \enumref{prop:HA.4.7.5.3}{1} can be computed by finite limits \cite[Proposition A.1]{arXiv:2207.09256}. 
	Hence, in this case, the diagram \eqref{diagam:Sh_of_Cech_nerve} also satisfies \enumref{prop:HA.4.7.5.3}{1}. 
\end{observation}

\begin{lemma}\label{lem:truncated_descent}
	Let $ n \geq 0 $ be an integer.
	Then:
	\begin{enumerate}[label=\stlabel{lem:truncated_descent}, ref=\arabic*]
		\item\label{lem:truncated_descent.1} The functor $ \Sh(-)_{\leq n} \colon \fromto{\LCHop}{\Catinfty} $ satisfies proper descent.

		\item\label{lem:truncated_descent.2} The functor $ \Sh(-)_{\leq n} \colon \fromto{\Compop}{\Catinfty} $ is a hypersheaf.
	\end{enumerate}
\end{lemma}

\begin{proof}
	For \enumref{lem:truncated_descent}{1}, note that since $ \Spc_{\leq n} $ is compactly generated, by \Cref{ex:descent_hypothesis_2,ex:descent_hypothesis_1} it suffices to check that for each proper surjection of locally compact Hausdorff spaces $ p \colon \surjto{X}{Y} $, the pullback functor $ \pupperstar \colon \fromto{\Sh(Y)_{\leq n}}{\Sh(X)_{\leq n}} $ is conservative. 
	This follows from the assumption that $ p $ is a surjection (\Cref{obs:pullback_conservativity}).

	Since $ \Sh(-)_{\leq n} $ factors through the full subcategory $ \Cat_{n+1} \subset \Catinfty $ spanned by the $ (n+1) $-cate\-gories, and $ \Cat_{n+1} $ is an $ (n+2) $-category, by \Cref{rmk:hypersheaves_in_an_n-category}, item \enumref{lem:truncated_descent}{2} is equivalent to the claim that $ \Sh(-)_{\leq n} $ is a sheaf.
	Thus \enumref{lem:truncated_descent}{2} follows from \enumref{lem:truncated_descent}{1} and the fact that the functor $ \Sh(-)_{\leq n} $ carries coproducts of topological spaces to products of \categories.
\end{proof}

\begin{corollary}\label{cor:general_descent}
	Let $ \E $ be a compactly assembled \category. 
	Then:
	\begin{enumerate}[label=\stlabel{cor:general_descent}, ref=\arabic*]
		\item\label{cor:general_descent.1} The functor $ \Shpost(-;\E) \colon \fromto{\LCHop}{\Catinfty} $ satisfies proper descent.

		\item\label{cor:general_descent.2} The functor $ \Shpost(-;\E) \colon \fromto{\Compop}{\Catinfty} $ is a hypersheaf.
	\end{enumerate}
\end{corollary}

\begin{proof}
	Since $ \E $ is compactly assembled, the functor $ \E \tensor (-) \colon \fromto{\PrL}{\PrL} $ commutes with limits of diagrams where the transition functors are left exact \cite[Lemma 2.15]{arXiv:2108.03545}.
	Hence it suffices to prove the claims for $ \E = \Spc $.
	In this case, the claims follow from \Cref{lem:limits_and_Postnikov_completion,lem:truncated_descent}.
\end{proof}




\subsection{Consequences in shape theory}\label{subsec:shape_theory}

We now record some consequences of \Cref{cor:general_descent} regarding the shape of compact Hausdorff spaces.
The reader unfamiliar with the shape is encouraged to consult \cites[\HTTsubsec{7.1.6}]{HTT}[\HAsec{A.1}]{HA}[\SAGsec{E.2}]{SAG}[Chapter 4]{arXiv:1807.03281}.
We follow the notations of \cite[Chapter 4]{arXiv:1807.03281}.

Since the shape $ \Shape \colon \fromto{\RTop}{\Pro(\Spc)} $ is a left adjoint, \Cref{cor:general_descent} implies:

\begin{corollary}\label{cor:Shape_of_proper_surjection}
	Let $ p \colon \surjto{X}{Y} $ be a proper surjection of locally compact Hausdorff spaces.
	Then the natural map of prospaces
	\begin{equation*}
		\colim_{[n] \in \Deltaop} \Shape\Shpost(X^{\cross_{Y} n}) \to \Shape\Shpost(Y)
	\end{equation*}
	is an equivalence.
\end{corollary}

\noindent Recall that the natural geometric morphism $ \fromto{\Shpost(X)}{\Sh(X)} $ induces an equivalences on protruncated shapes.
Combining \Cref{cor:Shape_of_proper_surjection} with work of Hoyois \cite{MR3763287}, we deduce that the protruncated shape preserves many (co)limits of compact Hausdorff spaces.

\begin{corollary}\label{cor:protrun_shapes_of_compacta}
	The protruncated shape $ \Shapeprotrun\Sh \colon \fromto{\Comp}{\Pro(\Spc_{<\infty})} $
	preserves:
	\begin{enumerate}[label=\stlabel{cor:protrun_shapes_of_compacta}, ref=\arabic*]
		\item\label{cor:protrun_shapes_of_compacta.1} Finite coproducts.

		\item\label{cor:protrun_shapes_of_compacta.2} Cofiltered limits and arbitrary products.

		\item\label{cor:protrun_shapes_of_compacta.3} Geometric realizations of Čech nerves of surjections.
	\end{enumerate}
\end{corollary}

\begin{proof}
	For \enumref{cor:protrun_shapes_of_compacta}{1}, note that the functor $ \Sh \colon \fromto{\Comp}{\RTop} $ preserves finite coproducts and the protruncated shape
	\begin{equation*}
		\Shapeprotrun \colon \fromto{\RTop}{\Pro(\Spc_{<\infty})}
	\end{equation*}
	is a left adjoint.
	Item \enumref{cor:protrun_shapes_of_compacta}{2} is \cite[Example 2.12]{MR3763287}, and \Cref{cor:Shape_of_proper_surjection} implies \enumref{cor:protrun_shapes_of_compacta}{3}.
\end{proof}

\begin{warning}
	The protruncated shape does \textit{not} preserve pullbacks of compact Hausdorff spaces.
	To see this, let $ \Sph{1} $ denote the topological circle, and choose a point $ t \in \Sph{1} $.
	The shapes of $ \Sh(\{t\}) $ and $ \Sh(\Sph{1}) $ are the point and the homotopy type $ \BZZ $, respectively (see \cites[\HTTthm{Corollary}{7.2.1.12}, \HTTthm{Theorem}{7.2.3.6} \& \HTTthm{Remark}{7.2.4.18}]{HTT}[Corollary 3.5]{arXiv:2010.06473}).
	The pullback of topological spaces $ \{t\} \cross_{\Sph{1}} \{t\} $ is a point, hence
	\begin{equation*}
		\Pi_{\infty}\Sh(\{t\} \cross_{\Sph{1}} \{t\}) \equivalent \ast \period
	\end{equation*}
	On the other hand,
	\begin{equation*}
		\Pi_{\infty}\Sh(\{t\}) \crosslimits_{\Pi_{\infty}\Sh(\Sph{1})} \Pi_{\infty}\Sh(\{t\}) \equivalent \ast \crosslimits_{\BZZ} \ast \equivalent \ZZ \period
	\end{equation*}
\end{warning}


\subsection{The necessity of Postnikov completion}\label{subsec:necesity_of_Postnikov_completion}

We conclude this section by explaining why Postnikov completion is necessary in the statement of \Cref{cor:general_descent}.
To do this, we give an example of a compact Hausdorff space $ K $ for which \smash{$ \Shhyp(K) $} is not Postnikov complete.%
We begin with some preliminary observations.

\begin{nul}
	Let $ \X $ be \atopos.
	For each $ m \geq 1 $, write $ \Kup_{\X}(\ZZ,m) $ for the constant object of $ \X $ at the Eilenberg--MacLane space $ \Kup(\ZZ,m) $.
	Note that since $ \Kup(\ZZ,m) $ is $ m $-truncated, $ \Kup_{\X}(\ZZ,m) $ is a hypercomplete object of $ \X $.
	Moreover, since the hypercomplete objects of $ \X $ are closed under limits, the object $ \prod_{m \geq 1} \Kup_{\X}(\ZZ,m) $ is also a hypercomplete object of $ \X $.
\end{nul}

\begin{observation}\label{obs:connected_product}
	If the \topos $ \X $ is Postnikov complete, then 
	\begin{equation*}
		\uptau_{\leq n}\Biggl(\prod_{m \geq 1} \Kup_{\X}(\ZZ,m) \Biggr) \equivalent \prod_{m = 1}^{n} \Kup_{\X}(\ZZ,m) \period
	\end{equation*}
	In particular, $ \uptau_{\leq 0}\paren{\prod_{m \geq 1} \Kup_{\X}(\ZZ,m)} $ is a terminal object of $ \X $, so that $ \prod_{m \geq 1} \Kup_{\X}(\ZZ,m) $ is connected.
\end{observation}

\begin{observation}\label{obs:connectedness_cohomology}
	Let $ X $ be a topological space and write $ F \colonequals \prod_{m \geq 1} \Kup_{\Sh(X)}(\ZZ,m) $.
	Then 
	\begin{equation*}
		\uppi_0 \Gamma(X;F) \equivalent \prod_{m \geq 1} \Hsheaf^m(X;\ZZ) \period
	\end{equation*}
	Note that if the sheaf $ F $ is connected, then any element $ (e_m)_{m \geq 1} \in \prod_{m \geq 1} \Hsheaf^m(X;\ZZ) $ \textit{vanishes locally} in the following sense: any point $ x \in X $ has an open neighborhood $ U \subset X $ such that for each $ m \geq 1 $, the class $ \restrict{e_m}{U} \in \Hsheaf^m(U;\ZZ) $ is zero.
\end{observation}

\begin{example}\label{ex:not_all_hypersheaves_are_Postnikov_complete}
	Let \smash{$ X \colonequals \prod_{m \geq 1} \Sph{m} $} be the product of positive-dimensional spheres.
	We claim that the \topos \smash{$ \Shhyp(X) $} is not Postnikov complete.
	In light of \Cref{obs:connected_product}, to see this it suffices to show that the hypersheaf $ \prod_{m \geq 1} \Kup_{\Sh(X)}(\ZZ,m) $ is not connected.
	\Cref{obs:connectedness_cohomology} shows that it suffices to construct cohomology classes \smash{$ e_i \in \Hsheaf^i(X;\ZZ) $} that do not simultaneously vanish on any nonempty open of $ X $.
	For this, let $ e_i $ be the pullback of the generator of
	\begin{equation*}
		\Hsheaf^i(\Sph{i};\ZZ) \isomorphic \ZZ
	\end{equation*}
	under the projection $ \pr_i \colon \fromto{\prod_{m \geq 1} \Sph{m}}{\Sph{i}} $.
\end{example}

\begin{example}\label{ex:not_all_hypersheaves_are_limits_of_their_Postnikov_towers}
	Not only is the \topos of hypersheaves on $ X = \prod_{m \geq 1} \Sph{m} $ not Postnikov complete, but there are also objects which are not limits of their Postnikov towers.
	To see this, consider the filtered colimit
	\begin{equation*}
		G \colonequals \colim\,\Biggl(\!\!
		\begin{tikzcd}[sep=1.5em]
			\Kup_{\Sh(X)}(\ZZ,1) \arrow[r] & \cdots \arrow[r] & \displaystyle\prod_{m = 1}^{n} \Kup_{\Sh(X)}(\ZZ,m) \arrow[r] & \displaystyle\prod_{m = 1}^{n+1} \Kup_{\Sh(X)}(\ZZ,m) \arrow[r] & \cdots
		\end{tikzcd}
		\!\!\Biggr)
	\end{equation*}
	formed in $ \Shhyp(X) $.
	Here, the transition map
	\begin{equation*}
		\prod_{m = 1}^{n} \Kup_{\Sh(X)}(\ZZ,m) \to \Kup_{\Sh(X)}(\ZZ,n+1) \cross \prod_{m = 1}^{n} \Kup_{\Sh(X)}(\ZZ,m)
	\end{equation*}
	is the constant map at the basepoint of $ \Kup_{\Sh(X)}(\ZZ,n+1) $ on the first factor and the identity on the second factor.
	Since connected objects are closed under filtered colimits, $ G $ is connected.
	The functor \smash{$ \trun_{\leq n} \colon \fromto{\Shhyp(X)}{\Shhyp(X)} $} preserves filtered colimits and finite products, hence
	\begin{align*}
		\trun_{\leq n} G &\equivalent \prod_{m = 1}^{n} \Kup_{\Sh(X)}(\ZZ,m) \period \\
	\intertext{As a consequence, we have}
		\lim_{n \geq 0} \trun_{\leq n} G &\equivalent \prod_{m \geq 1} \Kup_{\Sh(X)}(\ZZ,m) \period
	\end{align*}
	\Cref{ex:not_all_hypersheaves_are_Postnikov_complete} shows that the limit $ \lim_{n \geq 0} \trun_{\leq n} G $ is not connected.
	In particular, the natural map $ \fromto{G}{\lim_{n \geq 0} \trun_{\leq n} G} $ is not an equivalence.
\end{example}

\Cref{ex:not_all_hypersheaves_are_Postnikov_complete} lets us see that the presheaves $ \Sh, \Shhyp \colon \fromto{\Compop}{\Catinfty} $ are not sheaves.

\begin{lemma}
	Let $ \Scal \colon \fromto{\Compop}{\LTop} $ be a sheaf and assume that the restriction of $ \Scal $ to profinite sets factors through \smash{$ \LToppost $}.
	Then for each compact Hausdorff space $ K $, the \topos $ \Scal(K) $ is Postnikov complete.
\end{lemma}

\begin{proof}
	Since every compact Hausdorff space admits a surjection from a profinite set, this follows from the fact that the subcategory \smash{$ \LToppost \subset \LTop $} is closed under limits \Cref{rec:Postnikov_completion_preserves_limits_colimits}. 
\end{proof}

\noindent Since the functors $ \Sh $ and $ \Shpost $ agree on profinite sets (\Cref{ex:Sh_on_profinset}), \Cref{ex:not_all_hypersheaves_are_Postnikov_complete} shows:

\begin{corollary}\label{cor:Sh_Shhyp_not_sheaves}
	The presheaves $ \Sh, \Shhyp \colon \fromto{\Compop}{\LTop} $ are \textit{not} sheaves.
\end{corollary}


\section{The comparison functor}\label{sec:cohomology_comparison}

Let $ X $ be a topological space.
In this section, we construct a comparison geometric morphism
\begin{equation*}
	\cupperstar_{X} \colon \fromto{\Sh(X)}{\Sh(\Comp)_{/X}} \period
\end{equation*}
The idea is to define $ \cupperstar_X $ by left Kan extending the functor $ \yo \colon \fromto{\Open(X)}{\Sh(\Comp)_{/X}} $ to presheaves on $ X $ along the Yoneda embedding.
A priori, the right adjoint to this left Kan extension functor only lands in presheaves on $ X $.
To see that it factors through sheaves on $ X $ amounts to showing that the \topos $ \Sh(\Comp) $ has descent for open covers of topological spaces.
We prove this in \cref{subsec:descent_for_open_covers}.
\Cref{subsec:the_comparison_functor} constructs the comparison functor and gives an alternative description when $ X $ is a profinite set.
In \cref{subsec:naturality_of_the_comparison_functor}, we explain two naturality properties of the comparison functor.


\subsection{Descent for open covers}\label{subsec:descent_for_open_covers}

\begin{lemma}\label{lem:open_cover_effective_epi}
	Let $ X $ be a topological space and $ \{U_i\}_{i \in I} $ an open cover of $ X $.
	Then the natural map $ \fromto{\coprod_{i \in I} \yo(U_i)}{\yo(X)} $ is an effective epimorphism in $ \Sh(\Comp) $. 
\end{lemma}

\begin{proof}
	First we prove the claim under the assumption that $ X $ is compact Hausdorff.
	In this case, since $ X $ is compact, there exists a finite subset $ I_0 \subset I $ such that $ \{U_i\}_{i \in I_0} $ covers $ X $.
	Since $ X $ is compact Hausdorff, there exists a cover of $ X $ by closed subsets $ \{Z_i\}_{i \in I_0} $ such that $ Z_i \subset U_i $.
	Since $ \yo \colon \incto{\Comp}{\Sh(\Comp)} $ preserves finite coproducts and carries surjections to effective epimorphisms, the composite map
	\begin{equation*}
		\begin{tikzcd}
			\displaystyle\yo\Biggl(\coprod_{i \in I_0} Z_i \Biggr) \equivalent \coprod_{i \in I_0} \yo(Z_i) \arrow[r, hooked] & \displaystyle\coprod_{i \in I} \yo(U_i) \arrow[r] & \yo(X)
		\end{tikzcd}
	\end{equation*}
	is an effective epimorphism.
	Hence the second map $ \fromto{\coprod_{i \in I} \yo(U_i)}{\yo(X)} $ is also an effective epimorphism.

	Now we prove the claim in general.
	We need to show that for each compact Hausdorff space $ K $ and map $ \ftilde \colon \fromto{\yo(K)}{\yo(X)} $ in $ \Sh(\Comp) $, the induced map 
	\begin{equation*}
		\yo(K) \crosslimits_{\yo(X)} \coprod_{i \in I} \yo(U_i) \to \yo(K)
	\end{equation*}
	is an effective epimorphism.
	Since $ K $ is compact Hausdorff, by the Yoneda lemma the map $ \ftilde \colon \fromto{\yo(K)}{\yo(X)} $ is induced by a map of topological spaces $ f \colon \fromto{K}{X} $.
	Since coproducts in $ \Sh(\Comp) $ are universal and $ \yo \colon \fromto{\Top}{\Sh(\Comp)} $ preserves limits, we see that
	\begin{align*}
		\yo(K) \crosslimits_{\yo(X)} \coprod_{i \in I} \yo(U_i) &\equivalent \coprod_{i \in I} \yo(K) \crosslimits_{\yo(X)} \yo(U_i) \\ 
		&\equivalent \coprod_{i \in I} \yo(f^{-1}(U_i)) \period
	\end{align*}
	Since $ \{f^{-1}(U_i)\}_{i \in I} $ is an open cover of $ K $, the claim now follows from the compact case.
\end{proof}

\begin{corollary}\label{cor:descent_intersections}
	Let $ X $ be a topological space and let $ \{U_i\}_{i \in I} $ be an open cover of $ X $.
	Then the induced augmented simplicial object
	\begin{equation*}
		\begin{tikzcd}[sep=1.5em]
		    \cdots \arrow[r] \arrow[r, shift left=1.5ex] \arrow[r, shift right=1.5ex] & \displaystyle \coprod_{i_1, i_2 \in I} \yo(U_{i_1} \intersect U_{i_2}) \arrow[l, shift left=0.75ex] \arrow[l, shift right=0.75ex] \arrow[r, shift left=0.75ex] \arrow[r, shift right=0.75ex] & \coprod_{i \in I} \yo(U_i) \arrow[r] \arrow[l] & \yo(X)
		\end{tikzcd}
	\end{equation*}
	exhibits $ \yo(X) $ as the geometric realization of its restriction to $ \Deltaop \subset \Deltaplusop $.
\end{corollary}

\begin{proof}
	This follows from the fact that $ \fromto{\coprod_{i \in I} \yo(U_i)}{\yo(X)} $ is an effective epimorphism and the computation 
	\begin{align*}
		\Biggl( \coprod_{i \in I} \yo(U_i) \Biggr) \crosslimits_{\yo(X)} \cdots \crosslimits_{\yo(X)} \Biggl( \coprod_{i \in I} \yo(U_i) \Biggr) &\equivalent \coprod_{i_1,\ldots,i_n \in I} \yo(U_{i_1}) \crosslimits_{\yo(X)} \cdots \crosslimits_{\yo(X)} \yo(U_{i_n}) \\ 
		&\equivalent \coprod_{i_1,\ldots,i_n \in I} \yo(U_{i_1} \intersect \cdots \intersect U_{i_n}) \period \qedhere
	\end{align*}
\end{proof}

The usual cofinality argument shows: 

\begin{corollary}\label{cor:descent_covering_sieves}
	Let $ X $ be a topological space and let $ \Ucal $ be a covering sieve of $ X $.
	Then the natural map
	\begin{equation*}
		\fromto{\colim_{U \in \Ucal} \yo(U)}{\yo(X)}
	\end{equation*}
	is an equivalence in $ \Sh(\Comp) $.
\end{corollary}

Since the methods of proof are similar, we conclude this subsection showing that the functor $ \yo \colon \fromto{\Top}{\Sh(\Comp)} $ carries proper surjections to effective epimorphisms and preserves coproducts.
These results are not used in the rest of the paper.

\begin{lemma}
	Let $ p \colon \surjto{X}{Y} $ be a proper surjection of locally compact Hausdorff spaces.
	Then $ \yo(p) \colon \fromto{\yo(X)}{\yo(Y)} $ is an effective epimorphism in $ \Sh(\Comp) $.
\end{lemma}

\begin{proof}
	Since $ \yo $ preserves pullbacks and is fully faithful on locally compact Hausdorff spaces, we need to show that for every compact Hausdorff space $ K $ and map $ \fromto{K}{Y} $, the induced map
	\begin{equation*}
		\fromto{\yo(K \cross_Y X)}{\yo(K)}
	\end{equation*}
	is an effective epimorphism.
	Since $ p $ is a proper surjection, the pullback $ \pbar \colon \fromto{K \cross_Y X}{K} $ is also a proper surjection. 
	Since $ K $ is compact and $ \pbar $ is proper, $ K \cross_Y X $ is also compact.
	The fact that $ \yo \colon \incto{\Comp}{\Sh(\Comp)} $ carries surjections to effective epimorphisms completes the proof.
\end{proof}

\begin{lemma}
	The functor $ \yo \colon \fromto{\Top}{\Sh(\Comp)} $ preserves coproducts. 
\end{lemma}

\begin{proof}
	Let $ \{X_i\}_{i \in I} $ be a collection of topological spaces. 
	We need to show that for any compact Hausdorff space $ K $ and map $ f \colon \fromto{K}{\coprod_{i \in I} X_i} $, the induced map
	\begin{equation}\label{eq:yo_coproduct}
		\yo(K) \crosslimits_{\yo(\coprod_{j \in I} X_j)} \coprod_{i \in I} \yo(X_i) \to \yo(K)
	\end{equation}
	is an equivalence.
	Since $ \yo $ preserves pullbacks and coproducts in $ \Sh(\Comp) $ are universal, we see that
	\begin{align*}
		\yo(K) \crosslimits_{\yo(\coprod_{j \in I} X_j)} \coprod_{i \in I} \yo(X_i) &\equivalent \coprod_{i \in I} \yo(K) \crosslimits_{\yo(\coprod_{j \in I} X_j)} \yo(X_i) \\ 
		&\equivalent \coprod_{i \in I} \yo\Biggl( K \crosslimits_{\coprod_{j \in I} X_j} X_i \Biggr) \\ 
		&= \coprod_{i \in I} \yo(f^{-1}(X_i)) \period
	\end{align*}
	Since $ K $ is compact, there is a finite subset $ I_0 \subset I $ such that $ f^{-1}(X_i) \neq \emptyset $ if and only if $ i \in I_0 $.
	Hence
	\begin{align*}
		\yo(K) \crosslimits_{\yo(\coprod_{j \in I} X_j)} \equivalent \coprod_{i \in I_{0}} \yo(f^{-1}(X_i)) 
	\end{align*}
	and the map $ \fromto{\coprod_{i \in I_{0}} \yo(f^{-1}(X_i))}{\yo(K)} $ is induced by the inclusions $ \incto{f^{-1}(X_i)}{K} $.
	Note that
	\begin{equation*}
		K \isomorphic \coprod_{i \in I_{0}} f^{-1}(X_i) \period
	\end{equation*}
	Moreover, each $ f^{-1}(X_i) $ is clopen in $ K $, hence compact Hausdorff.
	Thus the fact that the Yoneda embedding $ \yo \colon \incto{\Comp}{\Sh(\Comp)} $ preserves finite coproducts shows that the map \eqref{eq:yo_coproduct} is an equivalence.
\end{proof}


\subsection{The comparison functor}\label{subsec:the_comparison_functor}

In this subsection, we define the comparison functor and give an alternative description of this functor in the case of a profinite set.

\begin{definition}
	Let $ X $ be a topological space.
	Write 
	\begin{equation*}
		\cupperstar_X \colon \fromto{\PSh(X)}{\Sh(\Comp)_{/X}} 
	\end{equation*}
	for the left Kan extension of the functor
	\begin{equation*}
		\yo \colon \fromto{\Open(X)}{\Sh(\Comp)_{/X}}
	\end{equation*}
	along the Yoneda embedding $ \incto{\Open(X)}{\PSh(X)} $.
\end{definition}

\begin{nul}
	Since the functor $ \yo \colon \fromto{\Open(X)}{\Sh(\Comp)_{/X}} $ is left exact, by \HTT{Proposition}{6.1.5.2}, the functor $ \cupperstar_{X} $ is also left exact.
\end{nul}

\begin{observation}
	The functor $ \cupperstar_X $ has a right adjoint
	\begin{equation*}
		c_{X,\ast} \colon \fromto{\Sh(\Comp)_{/X}}{\PSh(X)}
	\end{equation*}
	given by the formula
	\begin{equation*}
		c_{X,\ast}(G)(U) \colonequals \Map_{\Sh(\Comp)_{/X}}(\yo(U),G) \period
	\end{equation*}
	By \Cref{cor:descent_covering_sieves}, we see that for each $ G \in \Sh(\Comp)_{/X} $, the presheaf $ c_{X,\ast}(G) $ is a sheaf.
	Hence the adjunction $ \cupperstar_X \leftadjoint c_{X,\ast} $ restricts to a geometric morphism
	\begin{equation*}
		\begin{tikzcd}
			\Sh(X) \arrow[r, "\cupperstar_{X}", shift left] & \Sh(\Comp)_{/X} \period \arrow[l, "c_{X,\ast}", shift left]
		\end{tikzcd}
	\end{equation*}
\end{observation}

\begin{remark}
	If $ X $ is not compactly generated, then the functor
	\begin{equation*}
		\yo \colon \fromto{\Open(X)}{\Sh(\Comp;\Set)_{/X}}
	\end{equation*}
	need not be fully faithful. 
	Hence, the functor $ \cupperstar_X \colon \fromto{\Sh(X;\Set)}{\Sh(\Comp;\Set)_{/X}} $ need not be fully faithful.
\end{remark}

In the remainder of this subsection, we give an alternative description of the comparison geometric morphism when $ X $ is a profinite set.
To do so, we use the following description of $ \Sh(\Comp)_{/X} $ as sheaves on a site.

\begin{recollection}\label{rec:slice_topos}
	Let $ F \in \Sh(\Comp) $ be a sheaf of sets.
	Write \smash{$ \Comp_{/F} \subset \Sh(\Comp)_{/F} $} for the full subcategory spanned by morphisms from representable sheaves.
	Give $ \Comp_{/F} $ the finest Grothendieck toplogy that contains the sieves that become covering sieves after applying the forgetful functor \period{$ \fromto{\Comp_{/F}}{\Comp} $}.
	With respect to this topology, the natural functor
	\begin{equation*}
		\fromto{\Sh(\Comp_{/F})}{\Sh(\Comp)_{/F}}
	\end{equation*}
	is an equivalence of \categories \cite[Exposé III, \S5, Proposition 5.4]{MR50:7130}.
\end{recollection}

\begin{nul}
	Let $ X $ be a topological space.
	To simplify notation for sheaves with coefficients, we implicitly make use of the identification
	\begin{equation*}
		\Sh(\CompX) \equivalent \Sh(\Comp)_{/X}
	\end{equation*}
	provided by \Cref{rec:slice_topos} repeatedly throughout the rest of these notes.
\end{nul} 

\begin{recollection}\label{obs:Clop_basis_for_profinite}
	Let $ S $ be a profinite set, and write \smash{$ \Clop(S) \subset \Open(S) $} for the poset of clopen subsets of $ S  $.
	Then $ \Clop(S) $ is a basis closed under finite intersection.
	Hence for any presentable \category $ \E $, restriction along the inclusion $ \Clop(S)^{\op} \subset \Open(S)^{\op} $ defines an equivalence of \categories
	\begin{equation*}
		\Sh(S;\E) \equivalence \Sh(\Clop(S);\E) \period
	\end{equation*} 
	The inverse equivalence is given by right Kan extension \cites[Corollary A.8]{arXiv:2001.00319}[Corollary 3.12.14]{arXiv:1807.03281}.
\end{recollection}

\begin{nul}
	For a profinite set $ S $, the inclusion $ \Clop(S) \subset \CompS $ is a morphism of sites.
\end{nul}

\begin{observation}\label{obs:cupperstar_induce_by_morphism_of_sites_for_profinite_sets}
	Let $ S $ be a profinite set and $ \E $ a presentable \category.
	The composite
	\begin{equation*}
		\begin{tikzcd}[sep=3em]
			\Sh(\CompS) \equivalent \Sh(\Comp)_{/S} \arrow[r, "c_{S,\ast}"] & \Sh(S;\E) \arrow[r, "\sim"{yshift=-0.25em}, "\textup{restrict}"'] & \Sh(\Clop(S);\E) 
		\end{tikzcd}
	\end{equation*}
	carries a sheaf \smash{$ G \colon \fromto{(\CompS)^{\op}}{\Spc} $} to its restriction to $ \Clop(S)^{\op} \subset (\Comp_{/S})^{\op} $.
	As a consequence, the geometric morphism
	\begin{equation*}
		\begin{tikzcd}
			\Sh(S) \arrow[r, "\cupperstar_{S}", shift left] & \Sh(\CompS) \period \arrow[l, "c_{S,\ast}", shift left]
		\end{tikzcd}
	\end{equation*}
	is induced by the inclusion of sites $ \Clop(S) \subset \CompS $.
\end{observation}


\subsection{Naturality of the comparison functor}\label{subsec:naturality_of_the_comparison_functor}

\begin{lemma}\label{lem:c_natural_in_pullbacks}
	Let $ f \colon \fromto{X}{Y} $ be a map of topological spaces.
	Then the square
	\begin{equation}\label{square:c_naturality}
		\begin{tikzcd}[row sep=2em, column sep=3.5em]
			\Sh(Y) \arrow[r, "\cupperstar_Y"] \arrow[d, "\fupperstar"'] & \Sh(\Comp)_{/Y} \arrow[d, "X \cross_Y (-)"] \\
			\Sh(X) \arrow[r, "\cupperstar_X"'] & \Sh(\Comp)_{/X}
		\end{tikzcd} 
	\end{equation}
	canonically commutes.
\end{lemma}

\begin{proof}
	Since both composites are left adjoints and $ \Sh(Y) $ is generated under colimits by $ \Open(Y) $, it suffices to show that these composites agree on opens.
	For this, note for $ V \subset Y $ open, we have
	\begin{align*}
		\cupperstar_X\fupperstar(V) &= \yo(f^{-1}(V)) && \textup{(by definition)} \\
		&\equivalent \yo(V) \crosslimits_{\yo(Y)} \yo(X) && (\yo \textup{ preserves pullbacks)}\\
		&= \cupperstar_Y(V) \cross_{Y} X && \textup{(by definition)\phantom{pullbacks}} \period \qedhere
	\end{align*}
\end{proof}

\begin{nul}
	In light of \Cref{lem:c_natural_in_pullbacks}, the comparison functors assemble into a natural transformation $ \fromto{\Sh(-)}{\Sh(\Comp)_{/(-)}} $ of presheaves of \categories on $ \Top $.
\end{nul}

For open embeddings, the square \eqref{square:c_naturality} is also vertically left adjointable:

\begin{lemma}\label{lem:functoriality_of_comparison_for_opens}
	Let $ j \colon \incto{U}{X} $ be an open embedding of topological spaces.
	Then the square
	\begin{equation*}
		\begin{tikzcd}[row sep=2em, column sep=3.5em]
			\Sh(U) \arrow[r, "\cupperstar_U"] \arrow[d, "\jlowershriek"', hooked] & \Sh(\Comp)_{/U} \arrow[d, "\textup{forget}", hooked] \\
			\Sh(X) \arrow[r, "\cupperstar_X"'] & \Sh(\Comp)_{/X}
		\end{tikzcd} 
	\end{equation*}
	canonically commutes.
\end{lemma}

\begin{proof}
	Since both composites are left adjoints and $ \Sh(U) $ is generated under colimits by the opens of $ U $, it suffices to show that these composites agree on opens.
	For this, note that for $ V \subset U $ open, by definition we have
	\begin{equation*}
		\cupperstar_{X}\jlowershriek(V) = \cupperstar_X(V) = \yo(V) = \cupperstar_U(V)
	\end{equation*}
	as objects of $ \Sh(\Comp)_{/X} $.
\end{proof}


\section{Full faithfulness of the comparison functor}\label{sec:full_faithfulness_of_comparison}

Let $ \E $ be a compactly assembled \category (\Cref{rec:compactly_assembled}).
In this section, we show that if $ X $ is a locally compact Hausdorff space, then the comparison functor
\begin{equation*}
	\cupperstarpost \colon \fromto{\Shpost(X;\E)}{\Shpost(\CompX;\E)}
\end{equation*}
is fully faithful (\Cref{cor:cupperstarpost_fully_faithful_LCH}).
Since cohomology is computed by global sections, this implies that the sheaf cohomology and condensed cohomology of $ X $ agree (see \Cref{cor:cohomology_comparison_general,rmk:comparing_condensed_singular_sheaf}).

In \cref{subsec:full_faithfulness_extremally_disconnected}, we begin by proving the claim when $ X $ is an extremally disconnected profinite set.
In \cref{subsec:full_faithfulness_and_descent}, we use proper descent (\Cref{cor:general_descent}) and the functoriality of the comparison morphism to prove the result in general.
\Cref{subsec:cupperstar_Postnikov_completion_is_necessary} explains why the functor $ \cupperstar_X \colon \fromto{\Sh(X)}{\Sh(\Comp)_{/X}} $ is not generally fully faithful before Postnikov completion.


\subsection{Full faithfulness for extremally disconected profinite sets}\label{subsec:full_faithfulness_extremally_disconnected}

Let $ S $ be an extremally disconnected profinite set.
To prove that the functor \smash{$ \cupperstar_S $} is fully faithful, we verify that the inclusion \smash{$ \Clop(S) \subset \CompS $} satisfies the \textit{covering lifting property}.%
\footnote{The reader unfamiliar with the covering lifting property should consult \cite[Definition A.12]{arXiv:1803.01804}.}

\begin{recollection}
	If $ S $ is an extremally disconnected topological space, and $ U \subset S $ is open, then $ U $ is also extremally disconnected.
	However, closed subsets of extremally disconnected topological spaces need not be extremally disconnected \cite[Remark 2.4.11]{MR3379634}.
\end{recollection}

\begin{lemma}\label{lem:comp_covering_lifting_property}
	Let $ S $ be an extremally disconnected profinite set.
	Then the inclusion of sites 
	\begin{equation*}
		\incto{\Clop(S)}{\CompS}
	\end{equation*}
	satisfies the covering lifting property.
\end{lemma}

\begin{proof}
	Let $ U \subset S $ be a clopen subset and let $ \{f_i \colon \fromto{K_i}{U}\}_{i \in I} $ be a finite jointly surjective family in $ \CompS $.
	We claim that there exits a cover $ \{V_i\}_{i \in i} $ of $ U $ by clopens such that each inclusion $ \incto{V_i}{U} $ factors through $ f_i \colon \fromto{K_i}{U} $.
	To see this, first note that since $ U \subset S $ is clopen, $ U $ is also an extremally disconnected profinite set.
	Thus the surjection
	\begin{equation*}
		\coprod_{i \in I} f_i \colon \surjto{\coprod_{i \in I} K_i}{U}
	\end{equation*}
	admits a section $ s \colon \fromto{U}{\coprod_{i \in I} K_i} $.
	For each $ i \in I $, define $ V_i \colonequals s^{-1}(K_i) $.
	Since $ K_i $ is a clopen subset of $ \coprod_{i \in I} K_i $, the subset $ V_i \subset U $ is also clopen.
	Moreover, $ \{V_i\}_{i \in I} $ covers $ U $.
	Lastly, by construction, for each $ i \in I $, the composite
	\begin{equation*}
		\begin{tikzcd}
			V_i \arrow[r, "\restrict{s}{V_i}"] & K_i \arrow[r, "f_i"] & U
		\end{tikzcd}
	\end{equation*}
	coincides with the inclusion $ V_i \subset U $.
\end{proof}

The following is an application of \Cref{obs:cupperstar_induce_by_morphism_of_sites_for_profinite_sets}, \Cref{lem:comp_covering_lifting_property}, and \cite[Proposition A.13]{arXiv:1803.01804}.

\begin{notation}
	Let $ S $ be a profinite set and $ \E $ a presentable \category.
	Write 
	\begin{equation*}
		\cuppersharp_S \colon \incto{\PSh(\Clop(S);\E)}{\PSh(\CompS;\E)}
	\end{equation*}
	for the functor given by right Kan extension the inclusion \smash{$ \incto{\Clop(S)^{\op}}{(\CompS)^{\op}} $}.
\end{notation}

\begin{corollary}\label{cor:comp_covering_lifting_consequences}\label{cor:extremally_disconnected_cupperstar_fully_faithful}
	Let $ S $ be an extremally disconnected profinite set and $ \E $ a presentable \category.
	Then:
	\begin{enumerate}[label=\stlabel{cor:comp_covering_lifting_consequences}, ref=\arabic*]
		\item\label{cor:comp_covering_lifting_consequences.1} The functor $ \cuppersharp_S $ preserves sheaves.
		In particular, $ \cuppersharp_S $ restricts to a \emph{fully faithful} right adjoint to 
		\begin{equation*}
			c_{S,\ast} \colon \fromto{\Sh(\CompS;\E)}{\Sh(S;\E)} \period
		\end{equation*}
		
		\item\label{cor:comp_covering_lifting_consequences.2} The left adjoint $ \cupperstar_S \colon \incto{\Sh(S;\E)}{\Sh(\CompS;\E)} $ is fully faithful.
	\end{enumerate}
\end{corollary}


\subsection{Full faithfulness and descent}\label{subsec:full_faithfulness_and_descent}

In this subsection, we extend \enumref{cor:comp_covering_lifting_consequences}{2} to Postnikov complete sheaves on a locally compact Hausdorff space.
We begin by extending to compact Hausdorff spaces by proper descent.
To do so, we make use of the following well-known lemma:

\begin{lemma}\label{lem:limits_of_fully_faithful_functors}
	Let $ \Ical $ be \acategory, let $ \Xcal_{\bullet},\Ycal_{\bullet} \colon \fromto{\Ical}{\Catinfty} $ be diagrams of \categories, and let $ f_{\bullet} \colon \fromto{\Xcal_{\bullet}}{\Ycal_{\bullet}} $ be a natural transformation.
	If for each $ i \in \Ical $, the functor $ f_i \colon \fromto{\Xcal_i}{\Ycal_i} $ is fully faithful, then the induced functor on limits
	\begin{equation*}
		\lim_{i \in I} f_i \colon \lim_{i \in I} \Xcal_i \to \lim_{i \in I} \Ycal_i
	\end{equation*}
	is fully faithful.
\end{lemma}

\begin{corollary}\label{cor:Postnikov_completion_preserves_full_faithfulness}
	Let $ \fupperstar \colon \fromto{\Y}{\X} $ be a left exact left adjoint functor between \topoi.
	The following are equivalent:
	\begin{enumerate}[label=\stlabel{cor:Postnikov_completion_preserves_full_faithfulness}, ref=\arabic*]
		\item\label{cor:Postnikov_completion_preserves_full_faithfulness.1} The induced functor on Postnikov completions $ \fupperstarpost \colon \fromto{\Ypost}{\Xpost} $ is fully faithful.

		\item\label{cor:Postnikov_completion_preserves_full_faithfulness.2} For each integer $ n \geq 0 $, the restriction $ \fupperstar \colon \fromto{\Y_{\leq n}}{\X_{\leq n}} $ is fully faithful. 
	\end{enumerate}
\end{corollary}

\begin{proof}
	\Cref{nul:Postnikov_completion_equivalence_on_truncated} shows that \enumref{cor:Postnikov_completion_preserves_full_faithfulness}{1} $ \Rightarrow $ \enumref{cor:Postnikov_completion_preserves_full_faithfulness}{2}.
	\Cref{lem:limits_of_fully_faithful_functors} shows that \enumref{cor:Postnikov_completion_preserves_full_faithfulness}{2} $ \Rightarrow $ \enumref{cor:Postnikov_completion_preserves_full_faithfulness}{1}.
\end{proof}

\begin{corollary}\label{cor:extremally_disconnected_cupperstarpost_fully_faithful}
	Let $ S $ be an extremally disconnected profinite set and $ \E $ a compactly assembled \category.
	Then the comparison functor
	\begin{equation*}
		\cupperstarpost_S \colon \fromto{\Sh(S;\E)}{\Shpost(\CompS;\E)}
	\end{equation*}
	is fully faithful.
\end{corollary}

\begin{proof}
	Since \smash{$ \cupperstarpost \colon \fromto{\Sh(S)}{\Shpost(\CompS)} $} is left exact and $ \E $ is compactly assembled, by tensoring with $ \E $ it suffices to prove the claim for Postnikov sheaves valued in the \category of spaces \cite[Lemma 2.14]{arXiv:2108.03545}.
	The claim now follows from \Cref{cor:extremally_disconnected_cupperstar_fully_faithful,cor:Postnikov_completion_preserves_full_faithfulness}.
\end{proof}

\begin{recollection}[(van Kampen colimits)]\label{rec:van_Kampen_colimits}
	A colimit in \acategory $ \Xcal $ with pullbacks is \defn{van Kampen} if the functor $ \fromto{\Xcal^{\op}}{\Catinfty} $ given by $ \goesto{U}{\Xcal_{/U}} $ transforms it into a limit in $ \Catinfty $.
	A presentable \category $ \Xcal $ is \atopos if and only if all colimits in $ \Xcal $ are van Kampen; see \cites[\HTTthm{Proposition}{5.5.3.13}, \href{http://www.math.ias.edu/~lurie/papers/HTT.pdf\#theorem.6.1.3.9}{Theorem 6.1.3.9(3)}, \& \HTTthm{Proposition}{6.3.2.3}]{HTT}{MR3935451}.
\end{recollection}

\begin{corollary}\label{cor:cupperstarpost_fully_faithful_compact_Hausdorff}
	Let $ K $ be a compact Hausdorff space and $ \E $ a compactly assembled \category.
	Then the comparison functor
	\begin{equation*}
		\cupperstarpost_K \colon \fromto{\Shpost(K;\E)}{\Shpost(\CompK;\E)}
	\end{equation*}
	is fully faithful.
\end{corollary}

\begin{proof}
	Since the functors
	\begin{equation*}
		\goesto{K}{\Shpost(K;\E)} \andeq \goesto{K}{\Shpost(\CompK;\E)}
	\end{equation*}
	are hypersheaves on $ \Comp $ (\Cref{cor:general_descent,rec:van_Kampen_colimits}, respectively), by choosing a hypercover of $ K $ by extremally disconnected profinite sets and applying \Cref{lem:limits_of_fully_faithful_functors}, we are reduced to the case where $ K $ is extremally disconnected.
	We conclude by \Cref{cor:extremally_disconnected_cupperstarpost_fully_faithful}.
\end{proof}

The full faithfulness of $ \cupperstarpost_X $ is preserved by passing to open subspaces:

\begin{lemma}\label{lem:full_faithfulness_for_opens}
	Let $ j \colon \incto{U}{X} $ be an open embedding of topological spaces and let $ \E $ be a compactly assembled \category.
	If $ \cupperstarpost_X \colon \fromto{\Shpost(X)}{\Shpost(\Comp_{/X})} $ is fully faithful, then the functor
	\begin{equation*}
		\cupperstarpost_U \colon \fromto{\Shpost(U;\E)}{\Shpost(\Comp_{/U};\E)}
	\end{equation*}
	is also fully faithful. 
\end{lemma}

\begin{proof}
	Since \smash{$ \cupperstarpost_U \colon \fromto{\Sh(U)}{\Shpost(\Comp_{/U})} $} is left exact and $ \E $ is compactly assembled, by tensoring with $ \E $ it suffices to prove the claim for Postnikov sheaves valued in the \category of spaces \cite[Lemma 2.14]{arXiv:2108.03545}.
	In this case, \Cref{lem:functoriality_of_comparison_for_opens} shows that for each integer $ n \geq 0 $, the diagram
	\begin{equation*}
		\begin{tikzcd}[row sep=2em, column sep=3.5em]
			\Sh(U)_{\leq n} \arrow[r, "\cupperstar_U"] \arrow[d, "\jlowershriek"', hooked] & (\Sh(\Comp)_{/U})_{\leq n} \arrow[d, "\textup{forget}", hooked] \\
			\Sh(X)_{\leq n} \arrow[r, "\cupperstar_X"'] & (\Sh(\Comp)_{/X})_{\leq n}
		\end{tikzcd} 
	\end{equation*}
	commutes.
	Since $ \cupperstarpost_X $ is fully faithful, the bottom horizontal functor is fully faithful. 
	Hence the top horizontal functor is also fully faithful. 
	We conclude by \Cref{cor:Postnikov_completion_preserves_full_faithfulness}.
\end{proof}

Since every locally compact Hausdorff space embeds as an open in a compact Hausdorff space, \Cref{cor:cupperstarpost_fully_faithful_compact_Hausdorff,lem:full_faithfulness_for_opens} show:

\begin{corollary}\label{cor:cupperstarpost_fully_faithful_LCH}
	Let $ X $ be a locally compact Hausdorff space and $ \E $ a compactly assembled \category.
	Then the functor
	\begin{equation*}
		\cupperstarpost_X \colon \fromto{\Shpost(X;\E)}{\Shpost(\CompX;\E)}
	\end{equation*}
	is fully faithful.
\end{corollary}

We conclude this subsection with some cohomological consequences of \Cref{cor:cupperstarpost_fully_faithful_LCH}.
The following generalizes \cites[Theorem 3.11]{MR0448318}[Theorem 3.2]{Scholze:condensednotes}.

\begin{corollary}\label{cor:cohomology_comparison_general}
	Let $ X $ be a locally compact Hausdorff space, let $ R $ be a connective $ \Eup_1 $-ring spectrum, and let $ M $ be a bounded-above left $ R $-module spectrum. 
	Then the natural comparison map
	\begin{equation*}
		\fromto{\RGammasheaf(X;M)}{\RGammacond(X;M)}
	\end{equation*}
	is an equivalence in the \category $ \LMod(R) $ of left $ R $-module spectra.
\end{corollary}

\begin{remark}[(condensed, singular, and sheaf cohomology)]\label{rmk:comparing_condensed_singular_sheaf}
	Let $ R $ be a ring and $ X $ a topological space that is locally compact Hausdorff and locally weakly contractible.
	Write $ \Crm_{*}(X;R) \in \Dup(R) $ for the complex of \textit{singular chains} on $ X $.
	Given an object $ M \in \Dup(R) $, write 
	\begin{equation*}
		\Crm^{-*}(X;M) \colonequals \RHom_R(\Crm_{*}(X;R),M) \period
	\end{equation*}
	If $ M $ is an ordinary $ R $-module, then $ \Crm^{-*}(X;M) $ is what is usually referred to as the complex of \textit{singular cochains} on $ X $ with values in $ M $.

	If $ M $ is \tboundedabove, then \Cref{cor:cohomology_comparison_general} and \cite[Corollary 3.31]{arXiv:2010.06473} provide natural equivalences
	\begin{equation*}
		\RGammacond(X;M) \similarleftarrow \RGammasheaf(X;M) \equivalence \Crm^{-*}(X;M) \period
	\end{equation*}
	Hence the condensed, singular, and sheaf cohomologies of $ X $ all agree.
\end{remark}

\begin{nul}\label{nul:CW_comparing_condensed_singular_sheaf}
	Let $ X $ be a topological space that admits a locally finite CW structure.
	Since $ \Sh(X) $ is Postnikov complete, \Cref{cor:cohomology_comparison_general} and \cite[Corollary 3.31]{arXiv:2010.06473} actually imply that for \textit{any} object $ M \in \Dup(R) $, there are natural equivalences
	\begin{equation*}
		\RGammacond(X;M) \similarleftarrow \RGammasheaf(X;M) \equivalence \Crm^{-*}(X;M) \period
	\end{equation*}
\end{nul}


\subsection{The necessity of Postnikov completion}\label{subsec:cupperstar_Postnikov_completion_is_necessary}

We conclude by explaining why \Cref{cor:cupperstarpost_fully_faithful_LCH} is generally false before passing to Postnikov completions.
First we explain why \Cref{cor:cupperstarpost_fully_faithful_LCH} is false if $ \Shpost $ is replaced by $ \Sh $.

\begin{lemma}\label{lem:cupperstar_conservative_implies_hypercomplete}
	Let $ K $ be a compact Hausdorff space.
	Then the following are equivalent:
	\begin{enumerate}[label=\stlabel{lem:cupperstar_conservative_implies_hypercomplete}, ref=\arabic*]
		\item The functor $ \cupperstar_K \colon \fromto{\Sh(K)}{\Sh(\CompK)} $ is conservative.

		\item The \topos $ \Sh(K) $ is hypercomplete.
	\end{enumerate}
\end{lemma}

\begin{proof}
	Choose an extremally disconnected profinite set $ S $ equipped with a surjection $ p \colon \surjto{S}{K} $.
	Consider the commutative square
	\begin{equation*}
		\begin{tikzcd}[row sep=2em, column sep=3.5em]
			\Sh(K) \arrow[r, "\cupperstar_K"] \arrow[d, "\pupperstar"'] & \Sh(\Comp)_{/K} \arrow[d, "S \cross_K (-)"] \\
			\Sh(S) \arrow[r, "\cupperstar_S"'] & \Sh(\Comp)_{/S} \period
		\end{tikzcd} 
	\end{equation*}
	Since $ S $ is an extremally disconnected profinite set, $ \cupperstar_S $ is fully faithful (\Cref{cor:comp_covering_lifting_consequences}).
	Since $ p $ is surjective and the functor $ \Sh(\Comp)_{/(-)} $ satisfies descent for surjections of compact Hausdorff spaces, the right-hand vertical functor is conservative.
	Hence $ \cupperstar_K $ is conservative if and only if $ \pupperstar $ is conservative.
	To conclude, note that since $ p $ is surjective and $ \Sh(S) $ is hypercomplete, the pullback functor $ \pupperstar $ is conservative if and only if $ \Sh(K) $ is hypercomplete.
\end{proof}

\begin{warning}\label{warning:cupperstar_not_fully_faithful}
	Since there exist compact Hausdorff spaces $ K $ for which $ \Sh(K) $ is not hypercomplete \HTT{Counterexample}{6.5.4.8}, \Cref{lem:cupperstar_conservative_implies_hypercomplete} implies that the functor $ \cupperstar_K $ is \textit{not} generally fully faithful.
\end{warning}

Now we explain why \Cref{cor:cupperstarpost_fully_faithful_LCH} is false if $ \Shpost $ is replaced by $ \Shhyp $.

\begin{lemma}\label{lem:cupperstarhyp_fully_faithful_implies_Postnikov_complete}
	Let $ X $ be a topological space.
	If \smash{$ \cupperstarhyp_X \colon \fromto{\Shhyp(X)}{\Shhyp(\Comp)_{/X}} $} fully faithful, then every object of \smash{$ \Shhyp(X) $} is the limit of its Postnikov tower.
\end{lemma}

\begin{lemma}\label{lem:colocalization_of_Postnikov_complete}
	Let $ \fupperstar \colon \incto{\Y}{\X} $ be a fully faithful left exact left adjoint between \topoi.
	If every object of $ \X $ is the limit of its Postnikov tower, then every object of $ \Y $ is the limit of its Postnikov tower.
\end{lemma}

\begin{proof}	
	By \Cref{obs:pullback_to_Postnikov_completion_fully_faithful}, we need to show that \smash{$ \tupperstar_{\Y} \colon \fromto{\Y}{\Ypost} $} is fully faithful.
	Consider the commutative square 
	\begin{equation*}
		\begin{tikzcd}
			\Y \arrow[r, "\tupperstar_{\Y}"] \arrow[d, "\fupperstar"', hooked] & \Ypost \arrow[d, "\fupperstarpost"] \\ 
			\X \arrow[r, "\tupperstar_{\X}"'] & \Xpost \period
		\end{tikzcd}
	\end{equation*}
	By \Cref{cor:Postnikov_completion_preserves_full_faithfulness}, the functor $ \fupperstarpost $ is fully faithful, and by assumption $ \tupperstar_{\X} $ is fully faithful.
	Hence $ \tupperstar_{\Y} $ is also fully faithful. 
\end{proof}

\begin{proof}[Proof of \Cref{lem:cupperstarhyp_fully_faithful_implies_Postnikov_complete}]
	Combine \Cref{lem:colocalization_of_Postnikov_complete} with the fact that $ \Shhyp(\Comp)_{/X} $ is Postnikov complete \Cref{nul:ShhypComp_Postnikov_complete}.
\end{proof}

\begin{warning}\label{warning:cupperstarhyp_not_fully_faithful}
	\Cref{ex:not_all_hypersheaves_are_limits_of_their_Postnikov_towers,lem:cupperstarhyp_fully_faithful_implies_Postnikov_complete} show that for \smash{$ X = \prod_{m \geq 1} \Sph{m} $}, the functor \smash{$ \cupperstarhyp_X $} is \textit{not} fully faithful.
\end{warning}


\DeclareFieldFormat{labelnumberwidth}{#1}
\printbibliography[keyword=alph, heading=references]
\DeclareFieldFormat{labelnumberwidth}{{#1\adddot\midsentence}}
\printbibliography[heading=none, notkeyword=alph]

\end{document}